\documentclass[journal,twoside,web]{ieeecolor}
\usepackage{lcsys}
\usepackage{amsmath}
\usepackage{amsfonts}
\usepackage{amssymb}
\usepackage{graphicx}
\usepackage{color}
\usepackage{bbold} 
\usepackage{mathrsfs}
\usepackage{dsfont}
\usepackage[usenames,dvipsnames]{pstricks}
\allowdisplaybreaks
\newtheorem{lemma}{Lemma}
\newtheorem{remark}{Remark}
\newtheorem{theorem}{Theorem}

\newtheorem{assumption}{Assumption}

\usepackage{flushend}
\renewcommand{\baselinestretch}{0.9}

\newcommand{\modif}{\textcolor{black}}
\def\BibTeX{{\rm B\kern-.05em{\sc i\kern-.025em b}\kern-.08em
    T\kern-.1667em\lower.7ex\hbox{E}\kern-.125emX}}
\begin{document}

\title{Stabilization of Integral Delay Equations by solving Fredholm equations
}

\author{Jean Auriol$^{1}$, 
\thanks{$^{1}$Jean Auriol is with Universit\'e Paris-Saclay, CNRS, CentraleSup\'elec, Laboratoire des signaux et syst\`emes, 91190, Gif-sur-Yvette, France. {\tt\small jean.auriol@l2s.centralesupelec.fr}.}%
}
\maketitle
\thispagestyle{empty}

\begin{abstract}
In this paper, we design a stabilizing state-feedback control law for a system represented by a general class of integral delay equations subject to a pointwise and distributed input delay.
The proposed controller is defined in terms of integrals of the state and input history over a fixed-length time window. We show that the closed-loop stability is guaranteed, provided the controller integral kernels are solutions to a set of Fredholm equations. The existence of solutions is guaranteed under an appropriate spectral controllability assumption, resulting in an implementable stabilizing control law. The proposed methodology appears simpler and more \modif{general} compared to existing results in the literature. In particular, under additional regularity assumptions, the proposed approach can be expanded to address the degenerate case where only a distributed control term is present.
\end{abstract}

\section{Introduction}

\IEEEPARstart{I}{ntegral} Delay Equations (IDEs) are a class of linear difference equations with both pointwise and distributed delays. They naturally appear when modeling engineering or biological systems involving transport, communication, or measurement delays~\cite{niculescu2001delay}. Among examples of interest, we can cite sampled-data systems~\cite{fridman2014introduction}, population dynamics, or biomedical systems as epidemics~\cite{COOKE197687}. 
Interestingly,  linear first-order hyperbolic Partial Differential Equations (PDEs), which are widely used in the modeling of systems of balance laws~\cite{bastin2016stability}, can also be rewritten as IDEs~\cite{auriol2019explicit}. 

In this paper, we consider such a linear IDE subject to a pointwise and distributed input delay. This specific configuration naturally emerges when addressing the stabilization of under-actuated hyperbolic systems~\cite{auriol2020CDCunderactuation}, or networks of hyperbolic PDEs with actuators located at arbitrary nodes of the network~\cite{redaud2022stabilizing,redaud2021stabilizing}. For instance, this class of problems arises when developing traffic control strategies on vast road networks, where the actuator (ramp metering) can be located at a crossroad (junction of two roads)~\cite{espitia2022traffic}.

When designing stabilizing controllers for such systems, one of the difficulties to consider is that they combine pointwise and distributed delays on the state and the actuator. In many instances, such delays can lead to closed-loop system instability or result in suboptimal performance of the proposed output-feedback control laws. Furthermore, with the presence of a distributed delayed term in the control input, the stabilizability of the system is not guaranteed. Hence, necessary and sufficient stabilizability and controllability conditions \cite{halebook, bellman1963differential,mounier1998algebraic} have been derived in the literature using complex analysis.

Recently, a constructive approach has been introduced in~\cite{redaud2022stabilizing, redaud2021stabilizing} to stabilize the aforementioned class of IDEs. The proposed methodology involved employing a PDE system as a comparison system~\cite{niculescu2001delay} with equivalent stability properties. Subsequently, this PDE system was stabilized using the backstepping approach, incorporating a Fredholm transformation~\cite{yoshida1960lectures}. Notably, the invertibility of such a transformation, usually not guaranteed except under specific structural assumptions~\cite{coron2017finite, bribiesca2015backstepping}, was related to a spectral controllability condition. Although such a result was a breakthrough and could be seen as a possible generalization of the backstepping approach for underactuated systems, it has thus far only been developed for a scalar IDE featuring a single delay in the state and in the actuation. We believe it suffers from several limitations that may impede its generalizability to non-scalar systems with multiple delays. 1) First, it requires introducing a comparison system, adding potentially unnecessary complexity to the design process. 2) Then, demonstrating the existence and invertibility of the proposed Fredholm transformation entails intricate computations (such as integration along characteristic lines and rewriting of kernel equations as Fredholm integral equations), which could become exceedingly challenging for non-scalar systems with multiple delays. 3) Finally, the proposed methodology requires a non-zero pointwise delayed actuation term to avoid degenerate configurations.   

In this paper, we overcome the aforementioned limitations of~\cite{redaud2022stabilizing} by introducing a novel methodology for stabilizing IDEs with pointwise and distributed actuation. The proposed approach can be summarized as follows. We begin by considering a potential candidate for the controller, which is expressed in terms of integrals of the state and input history over a fixed-length time window.  We demonstrate that to guarantee closed-loop stability, the corresponding controller integral kernels need to satisfy a set of Fredholm integral equations. The existence of solutions to these equations is directly linked to a spectral controllability condition and is established using an operator framework~\cite{coron2016stabilization}. Compared to~\cite{redaud2022stabilizing}, the proposed approach features several advantages. It is simpler as it does not require introducing a comparison system \modif{to design the controller}, and the stabilization problem boils down to the existence of a solution to simple integral equations with a convolution structure. Then, we believe it easily extends to non-scalar systems with multiple delays. Lastly, under additional regularity assumptions, our approach can be expanded to address the degenerate case identified by~\cite{redaud2022stabilizing}.


 \emph{Notations:} 
 Consider $(a_0,a_1) \in \mathbb{R}^2$ such that $a_0\leq a_1$.  
 For all positive integer $n$, we denote $L^2([a_0,a_1],\mathbb{R}^n)$ the space of $\mathbb{R}^n$-valued functions that are square integrable on $[a_0,a_1]$. 
 We denote $C^{pw}_\tau=C^{pw}([-\tau,0],\mathbb{R}^n)$ the Banach space of piecewise continuous  functions mapping the interval $[-\tau,0]$ into $\mathbb{R}^n$ and denote its associated norm as $
||\phi_{[t]}||_{C^{pw}_\tau} = \sup_{s \in [-\tau,0]} \sqrt{\phi^T(t+s) \phi(t+s)}$.
 For a function $\phi: [-\tau, \infty) \mapsto \mathbb{R}^n$, its partial trajectory $\phi_{[t]}$ is defined by $\phi_{[t]}(\theta)=\phi(t+\theta), -\tau \leq \theta \leq 0$. 
The identity matrix of size $n \in \mathbb{N}$ is denoted $\text{Id}_n$. The index $n$ will be omitted if no confusion arises. We use $s$ for the Laplace variable.

\section{Problem under consideration} \label{sec_pb}

\subsection{Time-delay formulation}

Consider two known positive delays $\tau_0>0$ and $\tau_1>0$ and the initial data $x_0=x^0 \in C^{pw}_{\tau_0}$.  Let us introduce the following \modif{integral delay equation} (IDE)
\begin{align}
   & x(t)=~ax(t-\tau_0)+\int_0^{\tau_0} N(\nu)x(t-\nu)d\nu\nonumber \\
    &+bU(t-\tau_1)+ \int_0^{\tau_1} M(\nu)U(t-\nu)d\nu, \quad t \geq 0,\label{eq_neutral}
\end{align}
where $a \in (-1,1)$, $b\in \mathbb{R}$, and $N, M$ are two piecewise continuously differentiable functions, \modif{respectively defined on the intervals $[0,\tau_0]$ and $[0,\tau_1]$}. 
The function $U(t)$ is the input function and has values in $\mathbb{R}$. Its initial condition belongs to $C^{pw}_{\tau_1}$ and is denoted $U^0$.  
A function~$x:[-\tau_0, \infty) \rightarrow \mathbb{R}$ is called a solution of the initial value problem~\eqref{eq_neutral} if~$x_0=x^0$ and if equation~\eqref{eq_neutral} is satisfied for~$t \geq 0$. Since we have $|a|<1$, the open-loop system only has a finite number of unstable roots \cite{halebook}. More precisely, the \emph{principal part} of the system has to be exponentially stable. If such a condition is not fulfilled, then it is impossible to delay-robustly stabilize the system~\eqref{eq_neutral} \modif{(see \cite[Theorem 1.1]{logemann1996conditions})}. Due to the distributed delay term ($\int_0^{\tau_0}N(\nu) x(t-\nu)d\nu$), the open-loop system may however be unstable. The objective of this paper is to design a control law $U(t)$ that exponentially stabilizes the system in the sense of the $C^{pw}_{\tau_0}$-norm. \modif{Moreover, we want this control law to also be exponentially stable}. More precisely, we want to design a control law such that~$\exists~\nu > 0,~ C_0 > 0,~ \forall x^0 \in C_{\tau_0}^{pw}$, all solutions of the closed-loop \eqref{eq_neutral} satisfy $$\modif{||x_{[t]}||_{C^{pw}_{\tau_0}}+||U_{[t]}||_{C^{pw}_{\tau_1}} \leq C_0\mathrm{e}^{-\nu t}(||x^0||_{C^{pw}_{\tau_0}}+||U^0||_{C^{pw}_{\tau_1}})}.$$
The actuation term in~\eqref{eq_neutral} appears through pointwise and distributed delay terms. It has been seldom studied in the literature~\cite{bekiaris2016stability,ponomarev2015reduction}, and is a major difference compared to existing results. 
When we assume that there is at least a pointwise delay on the actuation (i.e., $b \ne 0$), a stabilizing controller was proposed in~\cite{redaud2021stabilizing,redaud2022stabilizing} using a PDE comparison system that induced intricate computations and the use of Fredholm integral transformations. 
We will show in this paper that such a comparison system is not required in the design. 
Moreover, the approach we propose in this paper extends to the case $b=0$ under additional regularity conditions.

\subsection{Design assumptions}
Besides the assumption $|a|<1$ that is required to avoid having an infinite number of unstable poles, we consider that $\tau_1\geq\tau_0$, i.e., the delay acting on the control input is larger than the state delay. 
Although our approach could straightforwardly be extended to cover the case $\tau_0>\tau_1$, we could also deal with this case by picking $U(t)=\hat U(t+\tau_1-\tau_0)$. 
Let us formally take the Laplace transform of equation~\eqref{eq_neutral} (with zero initial condition). We have $F_0(s) x(s)=F_1(s) U(s)$,
where the holomorphic function $F_0$ and $F_1$ are defined by
\begin{align}
   & F_0(s)=1-a\mathrm{e}^{-\tau_0s}-\int_0^{\tau_0}N(\nu)\mathrm{e}^{-\nu s}d\nu,\label{eq_F_0} \\
    &F_1(s)=b\mathrm{e}^{-\tau_1s}+\int_0^{\tau_1}M(\nu)\mathrm{e}^{-\nu s}d\nu.  \label{eq_F_1}
\end{align}
To guarantee the possibility of stabilizing system~\eqref{eq_neutral}, we make the following (spectral) controllability assumption that guarantees that $F_0$ and $F_1$ cannot simultaneously vanish. 
\begin{assumption} \label{Spectral controllability}\textbf{Spectral controllability}  \cite{mounier1998algebraic,pandolfi1976stabilization}\\
For all $s \in \mathbb{C}$, $\text{rank}[F_0(s), F_1(s)]=1$. 
\end{assumption}
\modif{We emphasize that Assumption~\ref{Spectral controllability} is not a necessary condition to stabilize the system~\eqref{eq_neutral}. It should be possible to use a weaker stabilizability condition, as ‘‘for all $s \in \mathbb{C}$, s.t. $\Re(s)\geq$ 0 $rank[F_0(s), F_1(s)]=1$". However, with this weaker assumption, the analysis would become more intricate as it is not directly possible to show the invertibility of the Fredholm integral operator introduced in Section~\ref{Sec_Result_b_ne_0}.}

  \subsection{A first result of Fredholm integral operator}
In this paper, we consider potential candidates for the controller 
expressed in terms of integrals of the state and input history over a fixed-length time window. 
We will show that the corresponding controller integral kernels have to satisfy a set of Fredholm integral
equations to guarantee closed-loop stability. To facilitate our analysis and establish the existence of solutions for these integral equations, we first present a result guaranteeing the invertibility of a Fredholm integral operator. Consider the Fredholm integral operator $\mathcal{T}: L^2([a_1,a_2], \mathbb{R}^n )\rightarrow  L^2([a_1,a_2], \mathbb{R}^n )$ defined by
\begin{align}
\mathcal{T} \begin{pmatrix} h(\cdot) \end{pmatrix}=M_0h(\cdot) - \int_a^b
K(\cdot,y) h(y)dy, \label{Chap_4_eq:fredholm_operator}
\end{align}
where $a_1<a_2$ are two real, $n>0$ is an integer, $M_0$ is an invertible matrix that belongs to $\mathbb{R}^{n\times n}$,  $K$ is bounded piecewise continuous function defined on the square $\{(\xi,y)\in [a_1,a_2]^2\}$. Note that the integral part of the operator has a regularizing effect, such that $\forall h \in L^2([a_1,a_2],\mathbb{R}^n),~\int_a^b K(x,y)h(y)dy \in H^1([a_1,a_2], \mathbb{R}^n)$. 

\begin{lemma}\label{Chap_4_lemma:operator_invert}
Consider two linear operators $\mathcal{A},\mathcal{B}$, such that $D(\mathcal{A})=D(\mathcal{B}) \subset  L^2([a_1,a_2], \mathbb{R}^n ) $. Consider the Fredholm integral operator $\mathcal{T}:  L^2([a_1,a_2], \mathbb{R}^n )\rightarrow L^2([a_1,a_2], \mathbb{R}^n )$ as defined by equation~\eqref{Chap_4_eq:fredholm_operator}. Assume that 
\begin{enumerate}
    \item $\ker(\mathcal{T}) \subset D(\mathcal{A})$,
    \item $\ker(\mathcal{T}) \subset \ker(\mathcal{B})$,
    \item $\forall h \in \ker(\mathcal{T}),~\mathcal{T}\mathcal{A} h=0 $,
    \item $\forall s \in \mathbb{C}$, $\ker (s\text{Id}-\mathcal{A}) \cap \ker (\mathcal{B})=\{0\}$.
\end{enumerate}
Then, the operator $\mathcal{T}$ is invertible, and \modif{its inverse is a Fredholm integral operator whose kernels inherit the same regularity properties}.
\end{lemma}
\begin{proof}
The proof can be found in~\cite[Lemma 2.2]{coron2016stabilization}.
 \end{proof} 
As it will appear in the paper, \modif{condition 4)} of Lemma~\ref{Chap_4_lemma:operator_invert} relates to the controllability condition given in Assumption~\ref{Spectral controllability}. 

\section{Design of a stabilizing controller when $b\ne 0$} \label{Sec_Result_b_ne_0}
In this section, we present our methodology to stabilize the system~\eqref{eq_neutral} in the case $b\ne 0$. We will show that the stabilization problem boils down to the
existence of a solution to simple integral equations with a convolution structure. The proposed approach appears to be simpler than the one proposed in~\cite{redaud2022stabilizing} as it does not require introducing a comparison
system and involves less intricate computations.
\subsection{Candidate control law and characteristic equation} \label{Sec_Candidate}
We look for the desired control law under the form
\begin{align} \label{Chap_4_eq_U_explicit}
    U(t)&=\int_{0}^{\tau_0} f(\nu) x(t-\nu) d\nu+ \int_{0}^{\tau_1} g(\nu)  U(t-\nu)  d\nu,
\end{align}
where $f$ and $g$ are piecewise continuously differentiable functions to be defined. Let us consider the extended variable $z(t)=(x(t),U(t))^T$. For all $t>0$, the state $z(t)$ verifies
\begin{align}
    &z(t)=\begin{pmatrix}a& 0 \\ 0 &0\end{pmatrix}z(t-\tau_0)+\int_0^{\tau_0} \begin{pmatrix}N(\nu)&0\\f(\nu)&0\end{pmatrix} z(t-\nu)d\nu \nonumber \\
    &+\begin{pmatrix}0& b\\ 0 &0\end{pmatrix}z(t-\tau_1)+\int_0^{\tau_1} \begin{pmatrix}0&M(\nu)\\0&g(\nu)\end{pmatrix}z(t-\nu) d\nu. \label{eq_vector_z}
\end{align}
Taking the Laplace transform of equation~\eqref{eq_vector_z}, we obtain $ z(s)=A(s)z(s)$, where $A(s)$ is defined as
\footnotesize
\begin{align*}
    A(s)=\begin{pmatrix}a\mathrm{e}^{-\tau_0s}+\int_0^{\tau_0}\mathrm{e}^{-\nu s}N(\nu)d\nu& b\mathrm{e}^{-\tau_1s}+\int_0^{\tau_1}\mathrm{e}^{-\nu s}M(\nu)d\nu \\ \int_0^{\tau_0}\mathrm{e}^{-\nu s}f(\nu)d\nu &\int_0^{\tau_1}\mathrm{e}^{-\nu s}g(\nu)d\nu\end{pmatrix}
\end{align*}
\normalsize
Therefore, the characteristic equation associated with the closed-loop system~\eqref{eq_vector_z} verifies for all $s\in \mathbb{C}$
\begin{align}
   0&=\det(\text{Id}-A(s))=1-a\mathrm{e}^{-\tau_0 s}-\int_0^{\tau_0}\mathrm{e}^{-\nu s}N(\nu)d\nu \nonumber \\
   &- \int_0^{\tau_1}\mathrm{e}^{-\nu s}g(\nu)d\nu-\int_0^{\tau_1}(\int_0^{\tau_0}\mathrm{e}^{-(\nu+\eta) s}f(\nu)d\nu)M(\eta)d\eta\nonumber \\
   &+a \int_0^{\tau_1}\mathrm{e}^{-(\nu+\tau_0) s}g(\nu)d\nu-b\int_0^{\tau_0}\mathrm{e}^{-(\nu+\tau_1) s}f(\nu)d\nu\nonumber \\
   &+\int_0^{\tau_0}(\int_0^{\tau_1}\mathrm{e}^{-(\nu+\eta) s}g(\nu)d\nu)N(\eta)d\eta. \label{eq_charac_brut}
\end{align}
To analyze the closed-loop stability of equation~\eqref{eq_vector_z}, it is equivalent to analyze the root location of equation~\eqref{eq_charac_brut}, as stated in the following lemma
\begin{lemma}
The closed-loop system~\eqref{eq_vector_z} is exponentially stable if and only if there exists~$\eta_0>0$ such that all solutions of the characteristic equation~\eqref{eq_charac_brut} satisfy~$\text{Re}(s)<-\eta_0$.
\end{lemma}
\begin{proof}
    The proof is in~\cite[Theorem 3.5]{halebook} and \cite{henry1974linear}.
\end{proof}
To analyze the root location of equation~\eqref{eq_charac_brut}, we first simplify it. Using Fubini's theorem, we can rewrite the double integrals. We have
\begin{align*}
    &\int_0^{\tau_0}(\int_0^{\tau_1}\mathrm{e}^{-(r+\eta) s}g(\eta)d\eta)N(r)dr=\color{black}\int_0^{\tau_1}g(\eta)\\
    &\color{black}(\int_\eta^{\eta+\tau_0}\mathrm{e}^{-\nu s}N(\nu-\eta)d\nu)d\eta=\color{black}\int_0^{\tau_0} (\int_0^\nu g(\eta)N(\nu\\
    &-\eta) d\eta)\mathrm{e}^{-\nu s}d\nu+\int_{\tau_0}^{\tau_1}(\int_{\nu-\tau_0}^\nu g(\eta)N(\nu-\eta)d\eta)\mathrm{e}^{-\nu s}d\nu\\
    &\quad+\int_{\tau_1}^{\tau_1+\tau_0}(\int_{\nu-\tau_0}^{\tau_1} g(\eta)N(\nu-\eta)d\eta)\mathrm{e}^{-\nu s}d\nu.
\end{align*}
We then Perform similar computations for the second integral. For all $s\in \mathbb{C}$, the characteristic equation~\eqref{eq_charac_brut} then rewrites
\begin{align}
   0=1-a\mathrm{e}^{-\tau_0 s}&-\int_0^{\tau_0}\mathrm{e}^{-\nu s}I_1(\nu)d\nu-\int_{\tau_0}^{\tau_1}\mathrm{e}^{-\nu s}I_2(\nu)d\nu\nonumber \\
   &-\int_{\tau_1}^{\tau_0+\tau_1}\mathrm{e}^{-\nu s}I_3(\nu)d\nu, \label{eq_charac_final}
\end{align}
where 
\begin{align}
        I_1(\nu)=&g(\nu)+N(\nu)+\int_0^\nu f(\eta)M(\nu-\eta)d\eta \nonumber \\
        &-\int_0^\nu g(\eta)N(\nu-\eta)d\eta,  \label{Chap_4_eq_def_I_1} \\
    I_2(\nu)=&g(\nu)-ag(\nu-\tau_0)+\int_0^{\tau_0} f(\eta)M(\nu-\eta)d\eta \nonumber \\
    &-\int_{\nu-\tau_0}^\nu g(\eta)N(\nu-\eta)d\eta, \label{Chap_4_eq_def_I_2} \\
    I_3(\nu)=&bf(\nu-\tau_1)+\int_{\nu-\tau_1}^{\tau_0} f(\eta)M(\nu-\eta)d\eta \nonumber \\
    &-ag(\nu-\tau_0)-\int_{\nu-\tau_0}^{\tau_1} g(\eta)N(\nu-\eta)d\eta,  \label{Chap_4_eq_def_I_3}
\end{align}
\normalsize
Provided that we can choose $f$ and $g$ such that $I_1\equiv 0$, $I_2\equiv 0$, and $I_3\equiv 0$ on their respective domain of definition, we would obtain the characteristic equation $1=a\mathrm{e}^{-\tau_0 s}$, which implies the exponential stability of $z$ since $|a|<1$. 
\modif{
\begin{remark}
    Our control approach consists of canceling the integral term $\int_0^{\tau_0} N(\nu) x(t-\nu)d\nu$ to obtain the exponentially stable system $x(t) = ax(t-\tau_0)$. Note that it should be possible to adjust the control strategy to modify the convergence rate. For instance, when $\tau_1=\tau_0$, we can first define the intermediate control input $\hat U(t)=U(t)+\frac{a-\bar a}{b}x(t)$, which will replace the coefficient $a$ by $\bar a$ (while also modifying the function $N$), before applying the integral cancellation procedure. However, we emphasize that this improvement in terms of convergence rate would be done at the cost of degraded robustness margins as shown in~\cite{auriol2020robust}. In the rest of the paper, we only focus on the cancellation of the integral term $\int_0^{\tau_0} N(\nu) x(t-\nu)d\nu$.
\end{remark}}

\subsection{Existence of the functions $f$ and $g$}
In this section, we prove that there exist $f$ and $g$ piecewise continuously differentiable functions such that the functions $I_1, I_2$ and $I_3$ defined in equations~\eqref{Chap_4_eq_def_I_1}-\eqref{Chap_4_eq_def_I_3} are identically equal to zero. More precisely, we have the following lemma.
\begin{lemma}\label{Chap_4_lemma_unique}
    Consider the functions $I_1$, $I_2$, and $I_3$ defined in equations~\eqref{Chap_4_eq_def_I_1}-\eqref{Chap_4_eq_def_I_3}. If $b\ne 0$, and if Assumption~\ref{Spectral controllability} is verified,  then there exist two unique piecewise continuously differentiable functions $(f,g)$ s.t. $I_1(\nu)=0$ for $\nu \in [0,\tau_0[$, $I_2(\nu)=0$ for $\nu \in [\tau_0,\tau_1[$, and $I_3(\nu)=0$ for $\nu \in [\tau_1,\tau_1+\tau_0]$.
\end{lemma}
\begin{proof}
    We first rewrite the delay $\tau_1$ as $ \tau_1=(n_0+1)\tau_0-\gamma,$ 
     where $n_0\in \mathbb{N}$ and $\gamma \in [0,\tau_0)$. Let us now 
    introduce, for all integer $1\leq k\leq n_0$, the intermediate functions $g_k$ defined for all $\nu \in (0,\tau_0]$ by $
        g_k(\nu)
        =$$g(\nu+k\tau_0-\gamma)$. 
        We also define the function $g_0$, for all $\nu \in (0,\tau_0]$, by \begin{align*}
            g_0(\nu)=\left\{
    \begin{array}{ll}
        g(\nu-\gamma) & \mbox{if $\nu \geq \gamma$}  \\
        0 & \mbox{if $\nu<\gamma$.}
    \end{array}
\right.
        \end{align*}
        We extend the functions $N$ (respectively $M$) by the value $0$ outside of the interval $[0,\tau_0]$ (respectively $[0,\tau_1]$).
     The system $I_1(\nu)=0, I_2(\nu)=0, I_3(\nu)=0$ rewrites as
    \begin{align}
    &-N(\nu-\gamma)=g_0(\nu)-\int_0^\nu g_0(\eta)N(\nu-\eta)d\eta\nonumber \\
    &\hspace{2cm}+\int_0^\nu f(\eta)M(\nu-\gamma-\eta)d\eta,\label{Chap_4_eq_g0}\\
    &-N(\nu+k\tau_0-\gamma)= g_k(\nu)-\int_0^{\nu} g_k(\eta)N(\nu-\eta)d\eta \nonumber \\
    &\hspace{1cm}-ag_{k-1}(\nu)-\int_{\nu}^{\tau_0} g_{k-1}(\eta)N(\nu-\eta+\tau_0)d\eta\nonumber \\
    &\hspace{2cm}+\int_0^{\tau_0} f(\eta)M(\nu+k\tau_0-\gamma-\eta)d\eta,\label{Chap_4_eq_gk} \\
    &0=bf(\nu)-ag_{n_0}(\nu) -\int_\nu^{\tau_0} g_{n_0}(\eta)N(\nu+\tau_0-\eta)d\eta \nonumber \\
    &\hspace{2cm}+\int_\nu^{\tau_0} f(\eta)M(\nu+\tau_1-\eta)d\eta,\label{Chap_4_eq_f}
    \end{align}
    where $1< k \leq n_0$ and $\nu \in (0,\tau_0]$. \modif{More precisely, equation $I_3(\nu)=0$ is equivalent to equation~\eqref{Chap_4_eq_f}. Equation~\eqref{Chap_4_eq_g0} (for $\nu \in [0,\tau_0]$)  is equivalent to $I_1(\nu)=0$ (for $\nu \in [0,\tau_0-\gamma]$), while equation~\eqref{Chap_4_eq_gk} (for $k=1$, $\nu \in [0,\gamma]$)  is equivalent to $I_1(\nu)=0$ (for $\nu \in (\tau_0-\gamma,\tau_0]$). Finally, equation~\eqref{Chap_4_eq_gk} (except for $k=1$, $\nu \in [0,\gamma]$) is equivalent to $I_2=0$. This shift between the two systems is due to the parameter $\gamma$.  }
    
    We now define the operator $\mathcal{T}:(L^2([0,\tau_0],\mathbb{R}^{n_0+2})) \rightarrow (L^2([0,\tau_0],\mathbb{R}^{n_0+2})) $ in equation~\eqref{eq_def_operator} \modif{(given in the next page)}. Note that the operator~$\mathcal{T}$ is a Fredholm operator as it satisfies equation~\eqref{Chap_4_eq:fredholm_operator}. 
    \begin{figure*}[!t]
\normalsize
\begin{align}
      ( \mathcal{T} \begin{pmatrix}
           g_0 \\ \vdots \\ g_{k}\\ \\\vdots  \\ f
       \end{pmatrix})(\nu)=  \begin{pmatrix}
       g_0(\nu)-\int_0^\nu g_0(\eta)N(\nu-\eta)d\eta+\int_0^\nu f(\eta)M(\nu-\gamma-\eta)d\eta \\
       \vdots\\
       g_k(\nu)-ag_{k-1}(\nu)-\int_{\nu}^{\tau_0} g_{k-1}(\eta)N(\nu-\eta+\tau_0)d\eta-\int_0^{\nu} g_k(\eta)N(\nu-\eta)d\eta\\+\int_0^{\tau_0} f(\eta)M(\nu+k\tau_0-\gamma-\eta)d\eta\\ 
       \vdots \\
           bf(\nu)-ag_{n_0}(\nu) -\int_\nu^{\tau_0} g_{n_0}(\eta)N(\nu+\tau_0-\eta)d\eta +\int_\nu^{\tau_0} f(\eta)M(\nu+\tau_1-\eta)d\eta 
       \end{pmatrix}.\label{eq_def_operator}
    \end{align}
\hrulefill
\end{figure*}
    We want to show that equations~\eqref{Chap_4_eq_g0}-\eqref{Chap_4_eq_f} admit a unique solution. This will result from the invertibility of the operator $\mathcal{T}$. To show this latter property,
 let us introduce the operators $A_\mathcal{T}$ defined by
    \begin{align}
&A_\mathcal{T}:~D(A_\mathcal{T}) \rightarrow L^2([0,\tau_0],\mathbb{R}^{n_0+2})\nonumber \\
&\begin{pmatrix}\psi_0\\\vdots \\ \psi_{n_0}\\  \phi \end{pmatrix}
\longmapsto \begin{pmatrix}\partial_x \psi_0+\phi(0)M(\cdot-\gamma)\\\vdots \\\partial_x \psi_{n_0} +\phi(0)M(\cdot+n_0\tau_0-\gamma)\\ \partial_x \phi +\phi(0)N(\cdot)  
\end{pmatrix}, \label{Chap_4_op_Ademo}
\end{align}
    where  $D(A_\mathcal{T})=\{\modif{(\psi_0,\dots,\psi_{n_{0}},\phi)} \in (H^1([0,\tau_0],\mathbb{R}^{n_{0}+2})),~\phi(\tau_0)=a\phi(0),~\psi_{n_{0}}(\tau_0)=b\phi(0),\psi_k(\tau_0)=\psi_{k+1}(0),~0\leq k<{n_{0}}\}$. We define the operator $B_\mathcal{T}:~D(A_\mathcal{T}) \rightarrow (L^2([0,\tau_0],\mathbb{R}^{{n_{0}}+2}))$, by 
    $$
B_{\mathcal{T}}(\begin{pmatrix}\psi_0& \cdots & \psi_{n_{0}}& \phi \end{pmatrix}^\top)=\psi_0(0).
$$
\modif{To find the operators~$A_\mathcal{T}$ and $B_\mathcal{T}$, we modified the PDE comparison system proposed in~\cite[eq. (13)-(14)]{redaud2021stabilizing} by adding extra states so all equations have the same velocities.}

We now show that the operators $\mathcal{T}$, $A_\mathcal{T}$ and $B_\mathcal{T}$ verify the requirements of Lemma~\ref{Chap_4_lemma:operator_invert}. Let us consider $h=(g_0, \dots, g_{n_{0}}, f)$ in $\ker(\mathcal{T})$.  \modif{We immediately get $g_0(0)=0$. Consequently, we obtain $g_1(0)=g_0(\tau_0)$.} We have for all $1\leq k < {n_{0}}$
\begin{align*}
    g_{k+1}(0)-g_k(\tau_0)=a(g_k(0)-g_{k-1}(\tau_0)).
\end{align*}
 We can then recursively show that $g_k(\tau_0)=g_{k+1}(0)$. Direct computations give $bf(0)=g_{n_0}(\tau_0)$ and $f(\tau_0)=af(0)$ (since $b\ne 0)$. Consequently $h \in D(A_\mathcal{T})$. Since  $g_0(0)=0$, we also have $h \in  \ker(B_\mathcal{T})$. 
We now need to show the third condition of Lemma~\ref{Chap_4_lemma:operator_invert}, i.e., for any $h\in \ker(\mathcal{T})$, we want to show that $\mathcal{T}(A_{\mathcal{T}}(h))=0$. Due to space restriction, we consider the case of continuously differentiable functions $N$ and $M$, but the proof can easily be adjusted to the case of a finite number of discontinuities. Let us compute the last component of $\mathcal{T}(A_{\mathcal{T}}(h))$. We have for all $\nu \in [0,\tau_0]$
\begin{align}
    &bf'(\nu)+bf(0)N(\nu)-ag_{n_0}'(\nu)-af(0)M(\nu+n_0\tau_0-\gamma)\nonumber \\
    &-\int_\nu^{\tau_0}((g_{n_0}'(\eta)+f(0)M(\eta+n_0\tau_0-\gamma))N(\nu+\tau_0-\eta)d\eta\nonumber \\
    &+\int_\nu^{\tau_0}(f'(\eta)+f(0)N(\eta))M(\nu+\tau_1-\eta))d\eta.
   \label{Chap_4_eq_f_calcul}
\end{align}
\modif{After integration by parts, this term rewrites
\begin{align*}
     &bf'(\nu)-ag_{n_0}'(\nu)+g_{n_0}(\nu)N(\tau_0)-f(\nu)M(\tau_1)\nonumber\\
    &+\int_\nu^{\tau_0}(f(\eta)M'(\nu+\tau_1-\eta)-g_{n_0}(\eta)N'(\nu+\tau_0-\eta))d\eta, 
\end{align*}
where we have used that $g_{n_0}(\tau_0)=bf(0)$ and $f(\tau_0)=af(0)$, since $g \in D(A_{\mathcal{T}})$.
}
In the meantime, since $h \in \ker(\mathcal{T})$, we also have \modif{
\begin{align*}
&0=bf'(\nu)-ag_{n_0}'(\nu)+g_{n_0}(\nu)N(\tau_0)-f(\nu)M(\tau_1)\\
&+\int_\nu^{\tau_0}(f(\eta)M'(\nu+\tau_1-\eta)-g_{n_0}(\eta)N'(\nu+\tau_0-\eta))d\eta.
\end{align*}}
Therefore, we obtain that the last component of $\mathcal{T}(A_{\mathcal{T}}(h))$ equals $0$.
 Performing analogous computations, for all the lines of $\mathcal{T}(A_{\mathcal{T}}(h))$, we obtain that $A_{\mathcal{T}}(h) \in \ker(\mathcal{T})$. We now need to prove the last requirement of Lemma~\ref{Chap_4_lemma:operator_invert}.
Consider $s\in\mathbb{C}$ and $h=(g_0,g_{n_0},\cdots,f) \in \ker(s\text{Id}-A_{\mathcal{T}}) \cap \ker(B_{\mathcal{T}})$. We have for all $0 \leq k \leq n_0$ and all $\nu \in [0,\tau_0]$, $sf (\nu)=f'(\nu)+f(0)N(\nu),$ and $sg_k(\nu)=g_k'(\nu)+f(0)M(\nu+k\tau_0-\gamma).$ Consequently, we obtain
    \begin{align}
    g_k(\nu)&=\mathrm{e}^{s\nu}g_k(0)-f(0) \int_0^{\nu}M(\eta+k\tau_0-\gamma)\mathrm{e}^{s(\nu-\eta)}d\eta, \nonumber \\
    f(\nu)&=(\mathrm{e}^{s\nu}-\int_0^{\nu}N(\eta)\mathrm{e}^{s(\nu-\eta)}d\eta)f(0), \label{eq_f_f0}
\end{align}
Since $f(\tau_0)=af(0)$, we obtain $af(0)=(\mathrm{e}^{s\tau_0}-\int_0^{\tau_0}N(\eta)\mathrm{e}^{s(\tau_0-\eta)}d\eta)f(0),$
which gives $F_0(s)f(0)=0$, where $F_0$ is defined in equation~\eqref{eq_F_0}. In the meantime, we obtain \small 
\begin{align*}
bf(0)&=g_{n_0}(\tau_0)\\
&=\mathrm{e}^{s\tau_0}g_{n_0}(0)-f(0)\int_0^{\tau_0}M(\eta+n_0\tau_0-\gamma)\mathrm{e}^{s(\tau_0-\eta)}d\eta \\
&=\mathrm{e}^{s\tau_0}g_{n_0-1}(\tau_0)-\modif{f(0)}\int_{n_0\tau_0-\gamma}^{\tau_1}M(\eta)\mathrm{e}^{s(\tau_1-\eta)}d\eta.
\end{align*}
\normalsize
Iterating the procedure, we obtain $F_1(s)f(0)$,  where $F_1$ is defined in equation~\eqref{eq_F_1}.  \modif{Due to Assumption~\ref{Spectral controllability}, we obtain that $f(0)=0$. Consequently, using equation~\eqref{eq_f_f0}, we have $h(\nu)=0$ for all $\nu \in [0,\tau_0]$,  and the last requirement of Lemma~\ref{Chap_4_lemma:operator_invert} is verified.}
 Consequently, the operator $\mathcal{T}$ is invertible and equations~\eqref{Chap_4_eq_g0}-\eqref{Chap_4_eq_f} admit a unique solution in $(L^2([0,\tau_0],\mathbb{R}^{n_0+2}))$. These solutions are piecewise continuously differentiable due to the regularity of the functions $M$ and $N$ and to the regularizing effect of the integral. 
\end{proof}
\modif{
\begin{remark}
    The functions $f$ and $g$ are bounded piecewise continuously differentiable functions that can be expressed in terms of the kernels $M$ and $N$ and the coefficients $a$ and $b$, using the operator $\mathcal{T}^{-1}$. However, their bounds can be significantly large if we are close to the critical case $b=0$.\end{remark} 
\begin{remark}
    We emphasize that Assumption~\ref{Spectral controllability} is necessary to prove Lemma~\ref{Chap_4_lemma_unique} (and in particular the last point of Lemma~\ref{Chap_4_lemma:operator_invert}) The proposed proof does not hold if the controllability condition is replaced by a weaker stabilizability condition.\end{remark} }
\subsection{Stabilizing control law}
Using Lemma~\ref{Chap_4_lemma_unique}, we can now write the following theorem.
\begin{theorem}
\label{Chap_4_State_feedback}
Assume that $b\ne 0$ and that Assumption~\ref{Spectral controllability} is verified.
Consider the functions $I_1$, $I_2$ and $I_3$ defined in~\eqref{Chap_4_eq_def_I_1}-\eqref{Chap_4_eq_def_I_3} and let $f$ and $g$ be the unique piecewise continuously differentiable functions that lead to $I_1(\nu)=0$ for all $\nu \in [0,\tau_0]$, $I_2(\nu)=0$ for all $\nu \in  ]\tau_0,\tau_1]$,  and $I_3(\nu)=0$ for $\nu \in  [\tau_1,\tau_0+\tau_1]$ (as stated in Lemma~\ref{Chap_4_lemma_unique}). Then the closed-loop system consisting of the plant~\eqref{eq_neutral} and the control law \eqref{Chap_4_eq_U_explicit} is exponentially stable. 
\end{theorem}
\begin{proof}
    The existence of $f$ and $g$ is implied by Lemma~\ref{Chap_4_lemma_unique}. For this choice of $f$ and $g$, the closed-loop characteristic equation rewrites $1=a\mathrm{e}^{-\tau_0 s}$ which implies the exponential stability of the state and the control law since $|a|<1$. 
\end{proof}
\modif{Taking the Laplace transform of equation~\eqref{Chap_4_eq_U_explicit}, we obtain
$U(s)=\frac{\int_0^{\tau_0}f(\nu)\mathrm{e}^{-\nu s}d\nu}{1-\int_0^{\tau_1}g(\nu)\mathrm{e}^{-\nu s}d\nu},$ which defines a strictly proper transfer function due to Riemann Lebesgues' lemma. Having a strictly proper control law guarantees the $w$-stability of the closed-loop system (see \cite[Th. 9.5.4]{curtain2012introduction}) and, consequently, the robustness to input delays and uncertainties on the parameters. It is also robust to admissible additive and multiplicative perturbations~\cite[Theorem 13]{auriol2023robustification}.}
\modif{
\section{Some insights to design a stabilizing controller when $b=0$}} \label{Sec_Result_b_0}
In this section, we consider the degenerate case $b=0$.
Due to space restrictions, we simplify our framework and assume $\tau_0=\tau_1=1$.  We will need the following additional assumption
\begin{assumption}
    The kernels $M$ and $N$ are two times continuously differentiable on $[0,1]$. Moreover, $aM(0)-M(1)\ne 0$ \modif{and $F_0(0) \ne 0$ (where $F_0$ is defined in eq.~\eqref{eq_F_0}).}\label{ass_regularity}
\end{assumption}
This assumption is conservative and is a current limitation of our approach in this critical configuration.
\modif{Let us now consider a stabilizing control law under the following form 
\begin{align} \label{Chap_4_eq_U_explicit_ter}
    U(t)&=\int_{0}^{1} f'(\nu) x(t-\nu) + g(\nu)  U(t-\nu)  d\nu\nonumber \\
    &+x(t)f(0)-x(t-1)f(1),
\end{align}
where $f$ and $g$ are piecewise continuously differentiable functions. Performing analogous computations to the ones done in Section~\ref{Sec_Candidate}, we obtain the characteristic equation~\eqref{eq_charac_final} with $I_2=0$ (since $\tau_1=\tau_0=1$) and where \begin{align}
        I_1(\nu)=&g(\nu)+M(0)f(\nu)+N(\nu)+\int_0^\nu f(\eta)M'(\nu-\eta)d\eta \nonumber \\
        &-\int_0^\nu g(\eta)N(\nu-\eta)d\eta, \label{eq_I_1_bis} \\
    I_3(\nu)=&-M(1)f(\nu-1)+\int_{\nu-1}^{1} f(\eta)M'(\nu-\eta)d\eta \nonumber \\
    &-ag(\nu-1)-\int_{\nu-1}^{1} g(\eta)N(\nu-\eta)d\eta. \label{eq_I_3_bis}
\end{align}
Again, we would like to find $f$ and $g$ such that the integral equations $I_1(\nu)=0$ and $I_3(\nu)=0$ admit a solution. As shown in the proof of Lemma~\ref{Chap_4_lemma_unique}, this can be done by showing that the operator $\mathcal{T}_1$ defined by equation~\eqref{eq_def_operator_bis} (given in the next page) is invertible. 
\begin{figure*}[!t]
    \footnotesize
         \begin{align}
      ( \mathcal{T}_1 \begin{pmatrix}
           g  \\ f
       \end{pmatrix})(\nu)=  \begin{pmatrix}
       g(\nu)+M(0)f(\nu)+\int_0^\nu f(\eta)M'(\nu-\eta)-g(\eta)N(\nu-\eta)d\eta \\ 
     -ag(\nu)-M(1)f(\nu)+\int_{\nu}^1 f(\eta)M'(1+\nu-\eta)d\eta-\int_{\nu}^1 g(\eta)N(1+\nu-\eta)d\eta
   \end{pmatrix}.\label{eq_def_operator_bis}
    \end{align}  
    \vspace{-0.3cm}
\hrulefill
\end{figure*}
The operator $\mathcal{T}_1$  rewrites $\mathcal{T}_1=K\text{Id}+\mathcal{I}$ where $\mathcal{I}$ is the (compact) integral component of the operator $\mathcal{T}_1$ and $K=\begin{pmatrix}
    1 & M(0)\\ -a & -M(1)
\end{pmatrix}$. The determinant of $K$ is equal to $-M(1)+aM(0) \ne 0$ due to Assumption~\ref{ass_regularity}. Therefore, the matrix $M$ is invertible, and $\mathcal{T}_1$ is a Fredholm integral operator. We can then easily adjust the proof of Lemma~\ref{Chap_4_lemma_unique} to show its invertibility. 
\begin{lemma}\label{Chap_4_lemma_unique_bbis}
    Consider the functions $I_1$ and $I_3$ defined in equations~\eqref{eq_I_1_bis}-\eqref{eq_I_3_bis}. 
    If Assumption~\ref{Spectral controllability} and Assumption~\ref{ass_regularity} are satisfied, then there exist two unique continuously differentiable functions $(f,g)$ such that $I_1(\nu)=0$ for $\nu \in [0,1[$, $I_3(\nu)=0$ for $\nu \in [1,2]$.
\end{lemma}
\begin{proof}
The proof follows the same steps as those of the proof of Lemma~\ref{Chap_4_lemma_unique}. Therefore, we only give a sketch of the proof.
Let us introduce  the operator $A_{\mathcal{T}_1}$ defined on $D(A_{\mathcal{T}_1}) \subset L^2([0,\tau_0],\mathbb{R}^{2})$ by
    \begin{align}
&A_{\mathcal{T}_1}:~D(A_{\mathcal{T}_1}) \rightarrow L^2([0,\tau_0],\mathbb{R}^{2})\nonumber \\
&\begin{pmatrix}\psi \\  \phi \end{pmatrix}
\longmapsto \begin{pmatrix}\partial_x \psi+\phi(0)M'(\cdot) \\ \partial_x \phi +\phi(0)N(\cdot)  
\end{pmatrix}, \label{Chap_4_op_A_bis}
\end{align}
    where  $D(A_{\mathcal{T}_1})=\{(\psi,\phi) \in (H^1([0,1],\mathbb{R}^2)),~\phi(1)=a\phi(0),~\psi(1)=-M(1)\phi(0)\}$. We define the operator $B_{\mathcal{T}_1}:~D(A_{\mathcal{T}_1}) \rightarrow (L^2([0,\tau_0],\mathbb{R}^{2}))$, by 
    $
B_{{\mathcal{T}_1}}(\begin{pmatrix}\psi& \phi \end{pmatrix}^\top)=M(0)\phi(0)+\psi(0).$ We can then show that the operators $\mathcal{T}_1$, $A_{\mathcal{T}_1}$ and $B_{\mathcal{T}_1}$ verify the requirements of Lemma~\ref{Chap_4_lemma:operator_invert} (using the fact that $aM(0)-M(1)\ne 0$). Note that Assumption~\ref{Spectral controllability} implies (using integration by parts) that for all $s\in \mathbb{C}\backslash\{0\}, \text{rank}[F_0(s),M(0)-M(1)\mathrm{e}^{-s}+\int_0^{1}M'(\nu)\mathrm{e}^{-\nu s}d\nu]=1$, and Assumption~\ref{ass_regularity} implies the property for $s=0$. 
\end{proof}
}
\begin{theorem}
\label{Chap_4_State_feedback_bis}
Assume that Assumption~\ref{Spectral controllability}  and Assumption~\ref{ass_regularity} are verified.  Consider the functions $I_1$, and $I_3$ defined in~\eqref{eq_I_1_bis}-\eqref{eq_I_3_bis} and let $f$ and $g$ be functions that lead to $I_1(\nu)=0$ for all $\nu \in [0,1]$, and $I_3(\nu)=0$ for $\nu \in  [1,2]$. Then, the closed-loop system consisting of the plant~\eqref{eq_neutral} and the control law \eqref{Chap_4_eq_U_explicit_ter} is exponentially stable. 
\end{theorem}
\modif{The proposed control strategy 
can be adjusted to deal with rationally dependent delays $\tau_0$ and $\tau_1$. For instance, when $\tau_1=(n_0+1)\tau_1$, the condition $M(1)\ne aM(0)$ in Assumption~\ref{ass_regularity} rewrites $M(1)\ne a^{n_0+1}M(0)$. Although we believe it is possible to adjust the approach to deal with non-rationally dependent delays, we have not been able to solve this case.}
\section{Simulation results} \label{Sec_simu}
In this section, we give some simulation results to illustrate our approach. We consider the system \eqref{eq_neutral}, with the coefficients  $\tau_0 = 1, \tau_1 = 1, a = 0.3, b=0$, and the integral coupling terms $N(\nu)=0.6+\sin(\pi*\nu)/5$ and $M(\nu)=\cos(\nu)$. \modif{The initial condition of the control input (on $[-\tau_1,0]$) is chosen as 0.} 
Assumption~\ref{Spectral controllability} is verified using a zero-location algorithm adjusted from the one presented in~\cite{saba2019stability}, while Assumption~\ref{ass_regularity} can be directly verified.
The control strategy is implemented using Matlab. The initial condition is a constant function $x^0 = 1$. The functions $f$ and $g$ are obtained by solving the Fredholm integral equations using a \modif{successive approximations approach (method of iterations). Although the convergence of this numerical approach is only guaranteed when the integral part of the Fredholm operator has a spectral radius smaller than one~\cite[Theorem 10.13]{kress1989linear}, we verified that for the considered example, the functions we numerically obtained were solutions of the system $I_1=I_3=0$. Among other possible numerical approaches to solve Fredholm integral equations, we can cite projection methods, collocation methods, Galerkin methods, or quadrature methods~\cite{kress1989linear}. These methods can possibly be numerically more efficient than the successive approximation approach. Moreover, we may obtain guarantees of convergence. In future work, we should investigate the properties of the kernels $f$ and $g$ better and focus on their numerical computation.} We have pictured in Figure~\ref{fig:norm_evol} the evolution of the state in open-loop and closed-loop. \modif{We also pictured the time evolution of the control input $U(t)$}. It can be seen that the proposed controller stabilizes the unstable open-loop system \modif{and that the control dynamics is also exponentially stable}. \modif{Finally, we have plotted the kernels $f$ and $g$  in Figure~\ref{fig:kernels}}.
\begin{figure}[ht]
\begin{center}
	\includegraphics[width=0.75\columnwidth]{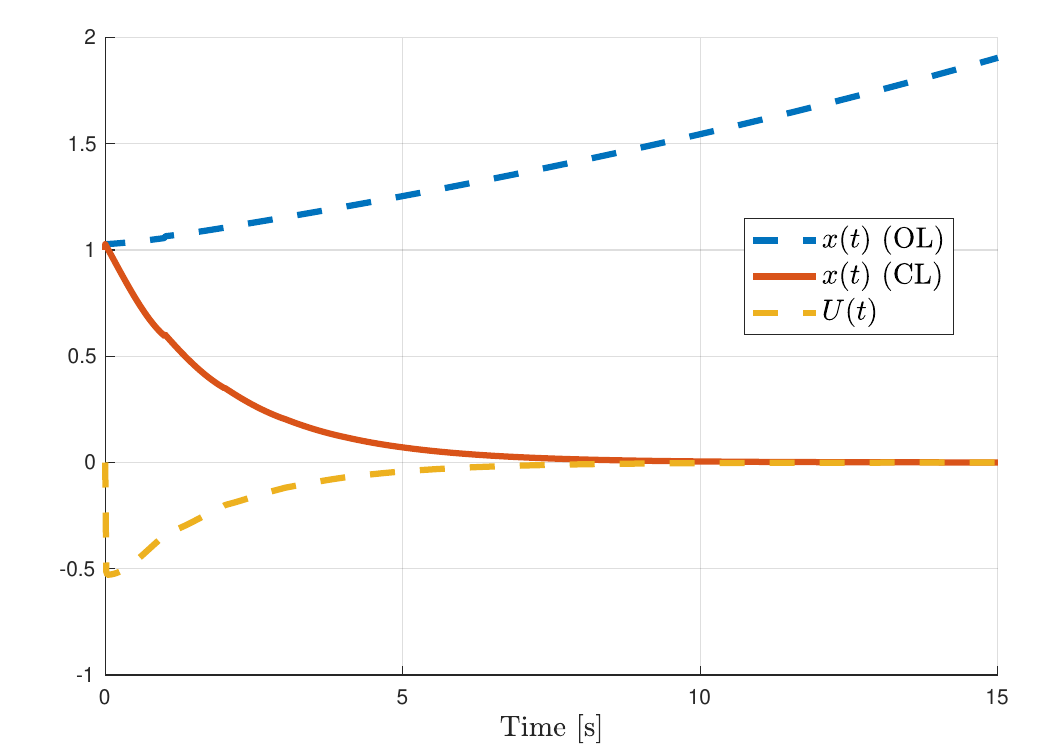}
	\caption{Evolution of the state of system \eqref{eq_neutral} in open-loop and closed-loop (with the corresponding control effort). }
	\label{fig:norm_evol}
	\end{center}
\end{figure}
\vspace{-0.9cm}
\begin{figure}[ht]
\begin{center}
	\includegraphics[width=0.75\columnwidth]{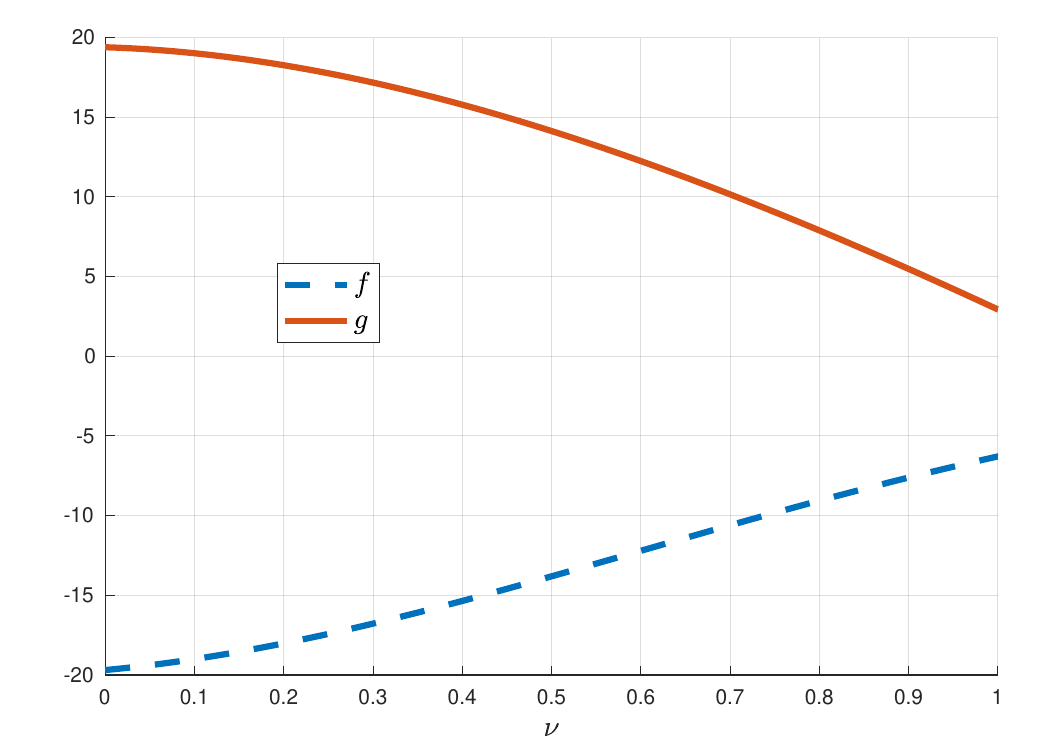}
	\caption{Representation of the kernels $f$ and $g$. }
	\label{fig:kernels}
	\end{center}
\end{figure}

\section{Concluding remarks}
We have proposed a new approach to stabilize a general class of integral delay equations. We introduced a candidate controller expressed as a distributed delayed feedback of the state and itself.
We showed that such a control law stabilizes the system if a set of Fredholm equations admitted a solution. Interestingly, spectral controllability conditions implied the existence of solutions for these
integral equations. The proposed approach overcomes the limitations of~\cite{redaud2022stabilizing} by proposing a simpler and more generalizable methodology. In future works, we will consider non-scalar systems with multiple delays. We will also endeavor to remove the conservative Assumption~\ref{ass_regularity}.

\bibliographystyle{plain}  
\bibliography{Biblio}

\begin{thebibliography}{10}

\bibitem{auriol2020CDCunderactuation}
J.~Auriol, F.~Bribiesca-Argomedo, and D.~Bresch-Pietri.
\newblock Stabilization of an underactuated 1+2 linear hyperbolic system with a
  proper control.
\newblock {\em Conference on Decision and Control}, 2020.

\bibitem{auriol2023robustification}
J.~Auriol, F.~Bribiesca~Argomedo, and F.~Di~Meglio.
\newblock Robustification of stabilizing controllers for {ODE}-{PDE}-{ODE}
  systems: a filtering approach.
\newblock {\em Automatica}, 147:110724, 2023.

\bibitem{auriol2019explicit}
J.~Auriol and F.~Di~Meglio.
\newblock An explicit mapping from linear first order hyperbolic {P}{D}{E}s to
  difference systems.
\newblock {\em Systems \& Control Letters}, 123:144--150, 2019.

\bibitem{auriol2020robust}
J.~Auriol and F.~Di~Meglio.
\newblock Robust output feedback stabilization for two heterodirectional linear
  coupled hyperbolic {PDE}s.
\newblock {\em Automatica}, 115:108896, 2020.

\bibitem{bastin2016stability}
G.~Bastin and J.-M. Coron.
\newblock {\em Stability and boundary stabilization of 1-{D} hyperbolic
  systems}, volume~88.
\newblock Springer, 2016.

\bibitem{bekiaris2016stability}
N.~Bekiaris-Liberis and M.~Krstic.
\newblock Stability of predictor-based feedback for nonlinear systems with
  distributed input delay.
\newblock {\em Automatica}, 70:195--203, 2016.

\bibitem{bellman1963differential}
R.~E. Bellman and K.~L. Cooke.
\newblock {\em Differential-difference equations}.
\newblock Rand Corporation, 1963.

\bibitem{saba2019stability}
D.~Bou~Saba, F.~Bribiesca-Argomedo, J.~Auriol, M.~Di~Loreto, and F.~Di~Meglio.
\newblock Stability analysis for a class of linear $2\times 2$ hyperbolic
  {PDE}s using a backstepping transform.
\newblock {\em IEEE Transactions on Automatic Control}, 65(7):2941--2956, 2019.

\bibitem{bribiesca2015backstepping}
F.~Bribiesca-Argomedo and M.~Krstic.
\newblock Backstepping-forwarding control and observation for hyperbolic {PDE}s
  with {F}redholm integrals.
\newblock {\em Automatic Control, IEEE Transactions on}, 60(8):2145--2160,
  2015.

\bibitem{COOKE197687}
K.~Cooke and J.~Kaplan.
\newblock A periodicity threshold theorem for epidemics and population growth.
\newblock {\em Mathematical Biosciences}, 31(1):87--104, 1976.

\bibitem{coron2016stabilization}
J.-M. Coron, L.~Hu, and G.~Olive.
\newblock Stabilization and controllability of first-order integro-differential
  hyperbolic equations.
\newblock {\em Journal of Functional Analysis}, 271(12):3554--3587, 2016.

\bibitem{coron2017finite}
J.-M. Coron, L.~Hu, and G.~Olive.
\newblock Finite-time boundary stabilization of general linear hyperbolic
  balance laws via {F}redholm backstepping transformation.
\newblock {\em Automatica}, 84:95--100, 2017.

\bibitem{curtain2012introduction}
R.~F. Curtain and H.~Zwart.
\newblock {\em An introduction to infinite-dimensional linear systems theory},
  volume~21.
\newblock Springer Science \& Business Media, 2012.

\bibitem{espitia2022traffic}
N.~Espitia, J.~Auriol, H.~Yu, and M.~Krstic.
\newblock Traffic flow control on cascaded roads by event-triggered output
  feedback.
\newblock {\em International Journal of Robust and Nonlinear Control},
  32(10):5919--5949, 2022.

\bibitem{fridman2014introduction}
E.~Fridman.
\newblock {\em Introduction to time-delay systems: Analysis and control}.
\newblock Springer, 2014.

\bibitem{halebook}
J.K. Hale and S.M. Verduyn~Lunel.
\newblock {\em Introduction to functional differential equations}.
\newblock Springer-Verlag, 1993.

\bibitem{henry1974linear}
D.~Henry.
\newblock Linear autonomous neutral functional differential equations.
\newblock {\em Journal of Differential Equations}, 15(1):106--128, 1974.

\bibitem{kress1989linear}
R.~Kress, V.~Mazya, and V.~Kozlov.
\newblock {\em Linear integral equations}, volume~82.
\newblock Springer, 1989.

\bibitem{logemann1996conditions}
H.~Logemann, R.~Rebarber, and G.~Weiss.
\newblock Conditions for robustness and nonrobustness of the stability of
  feedback systems with respect to small delays in the feedback loop.
\newblock {\em SIAM Journal on Control and Optimization}, 34(2):572--600, 1996.

\bibitem{mounier1998algebraic}
H.~Mounier.
\newblock Algebraic interpretations of the spectral controllability of a linear
  delay system.
\newblock In {\em Forum Mathematicum}, volume~10, pages 39--58. De Gruyter,
  1998.

\bibitem{niculescu2001delay}
S.-I. Niculescu.
\newblock {\em Delay effects on stability: a robust control approach}, volume
  269.
\newblock Springer Science \& Business Media, 2001.

\bibitem{pandolfi1976stabilization}
L.~Pandolfi.
\newblock Stabilization of neutral functional differential equations.
\newblock {\em Journal of Optimization Theory and Applications},
  20(2):191--204, 1976.

\bibitem{ponomarev2015reduction}
A.~Ponomarev.
\newblock Reduction-based robustness analysis of linear predictor feedback for
  distributed input delays.
\newblock {\em IEEE Transactions on Automatic Control}, 61(2):468--472, 2015.

\bibitem{redaud2021stabilizing}
J.~Redaud, J.~Auriol, and S.-I. Niculescu.
\newblock Stabilizing integral delay dynamics and hyperbolic systems using a
  {F}redholm transformation.
\newblock In {\em 2021 60th IEEE Conference on Decision and Control (CDC)},
  pages 2595--2600. IEEE, 2021.

\bibitem{redaud2022stabilizing}
J.~Redaud, J.~Auriol, and S.-I. Niculescu.
\newblock Stabilizing output-feedback control law for hyperbolic systems using
  a {F}redholm transformation.
\newblock {\em IEEE Transactions on Automatic Control}, 67(12):6651--6666,
  2022.

\bibitem{yoshida1960lectures}
K.~Yoshida.
\newblock {\em Lectures on differential and integral equations}, volume~10.
\newblock Interscience Publishers, 1960.

\end{thebibliography}

\end{document}